\documentclass[a4paper,11pt]{article}
\usepackage{amsmath,amsfonts,amssymb,amsthm}
\usepackage{graphics}
\usepackage{epsfig}
\usepackage{fullpage}
\usepackage{esint} % for avg int

\newtheorem{theorem}{Theorem}
\newtheorem{remark}{Remark}
\newtheorem{lemma}{Lemma}
\newtheorem{corollary}{Corollary}

\providecommand{\B}{\mathcal{B}}
\providecommand{\uoset}[3]{\underset{\phantom{#1}}{\overset{#2}{#3}}}
\newcommand{\en}[1]{\left< #1 \right>}
\newcommand{\p}[1]{\left(#1\right)}
\renewcommand{\a}[1]{\left| #1 \right|}
\providecommand{\Ah}{A_{\textrm{hom}}}
\newcommand{\R}{\mathbb{R}}

\newcommand{\ignore}[1]{}
\newcommand{\oh}{{\textstyle\frac{1}{2}}}
\providecommand{\ud}[1]{\, \mathrm{d} #1}
\providecommand{\dx}{\ud{x}}
\providecommand{\dy}{\ud{y}}
\providecommand{\dbx}{\ud{\bar x}}
\providecommand{\dby}{\ud{\bar y}}
\providecommand{\F}{F_\rho}

\author{Peter Bella \and Arianna Giunti}
\title{Green's function for elliptic systems: moment bounds}

\begin{document}

\date{}
\maketitle

\centerline{Max Planck Institute for Mathematics in the Sciences}
\centerline{Inselstrasse 22, 04103 Leipzig, Germany}

\begin{abstract}
 We study estimates of the Green's function in $\R^d$ with $d \ge 2$, for the linear second order elliptic equation in divergence form with variable uniformly elliptic coefficients. In the case $d \ge 3$, we obtain estimates on the Green's function, its gradient, and the second mixed derivatives which scale optimally in space, in terms of the ``minimal radius'' $r_*$ introduced in [Gloria, Neukamm, and Otto: A regularity theory for random elliptic operators; ArXiv e-prints (2014)]. As an application, our result implies optimal stochastic Gaussian bounds in the realm of homogenization of equations with random coefficient fields with finite range of dependence. In two dimensions, where in general the Green's function does not exist, we construct its gradient and show the corresponding estimates on the gradient and mixed second derivatives. Since we do not use any scalar methods in the argument, the result holds in the case of uniformly elliptic systems as well.
\end{abstract}

\section{Introduction}

This paper is a contribution to recently very active area of quantitative stochastic homogenization of second order uniformly elliptic operators, the main goal of which is to quantify how close is the large scale behavior of the heterogeneous operator $(-\nabla \cdot A(x) \nabla)^{-1}$ to the behavior of the constant-coefficient solution operator $(-\nabla \cdot \Ah \nabla)^{-1}$. Here $A(x)$ stands for the non-constant (random) coefficient field defined on $\R^d$ and $\Ah$ is called the matrix of homogenized coefficients. 

As originally realized in their seminal papers by Papanicolaou and Varadhan~\cite{papvar} and, independently, by Kozlov~\cite{kozlov}, the central object in the homogenization of elliptic operators with random coefficients is the \emph{corrector} $\phi_\xi$, defined for each direction $\xi \in \R^d$ as a solution of the following elliptic problem
\begin{equation}\nonumber
 -\nabla_x \cdot (A(x) \nabla_x (x \cdot \xi + \phi_\xi(A,x))) = 0
\end{equation}
in the whole space $\R^d$. The function $\phi$ is called corrector since it corrects the linear function $x \cdot \xi$, which is clearly solution to the constant-coefficient equation, to be a solution of the equation with heterogeneous coefficients. Since $\phi$ serves as a correction of a linear function, it should naturally be smaller, i.e., sublinear. Assuming the distribution of random coefficient fields $A$ is stationary (meaning the joint distribution of $A$ at any two points in $\R^d$ is the same) and ergodic (meaning any shift-invariant random variable is almost surely constant, a property encoding decorrelation of coefficient fields over large scales), they showed that correctors are almost surely sublinear and can be used to define the homogenized coefficient
\begin{equation}\nonumber%\label{ahom}
 \Ah e_i := \en{ A(e_i + \nabla \phi_{e_i}) }.
\end{equation}
Since the problem is linear, it clearly suffices to study the $d$ correctors $\phi_i := \phi_{e_i}$ for $i=1,\ldots,d$. Second, borrowing notation from the statistical physics, $\en{\cdot}$ stands for the ensemble average (expected value) with respect to a probability distribution on the space of coefficient fields $A$. Here and also later, we will often drop the argument $A$ in random quantities like the corrector as well as the argument $x$ related to the spatial dependence of quantities like the coefficient field $A$ or the corrector $\phi$. 

Both mentioned works~\cite{kozlov, papvar} were purely qualitative in the sense that they showed the sublinearity of the corrector in the limit of large scales without any rate. Assuming the correlation of the coefficient fields decays with a specific rate (either encoded by some functional inequality like the Spectral Gap estimate or the Logarithmic Sobolev Inequality, or by some mixing conditions or even assuming finite range of dependence), one goal of quantitative theory is to \emph{quantify} the sublinearity (smallness) of the corrector and consequences thereof. 

Though the present result is purely deterministic in the sense that it translates the fact that the energy of any $A$-harmonic function satisfies ``mean value property'' from some scale on (a fact that follows from the sublinearity of the corrector, see \cite{GNO4}) into estimates on Green's function and its derivatives, we will first mention some recent results related to sublinearity of the corrector without going too much into details. 

In~\cite{GNO4}, Gloria, Neukamm, and Otto introduced the random variable $r_* = r_*(A)$ called \emph{minimal radius}, which for given fixed $\delta = \delta(d,\lambda) > 0$ (here $\lambda$ denotes the ellipticity contrast) is defined as
\begin{equation}\label{defrstar}
 r_* := \inf \biggl\{ r \ge 1, \textrm{ dyadic } : \forall R \ge r, \textrm{ dyadic }: \frac{1}{R^2} \fint_{B_R} \biggl|(\phi,\sigma) - \fint_{B_R} (\phi,\sigma)\biggr|^2 \le \delta \biggr\}.
\end{equation}
Here $(\phi,\sigma)$ stands for the augmented corrector, where the new element $\sigma$ (called vector potential) can be used for obtaining good error estimates and which was originally introduced for the periodic homogenization (see, e.g., \cite{AL1}). Here and in what follows $B_R$ stands for a ball of radius $R$ centered at the \emph{origin} and $\fint$ denotes the average integral. In~\cite{GNO4} they showed that for small enough $\delta=\delta(d,\lambda)$ the sublinearity of the corrector implies the mean value property, meaning that for $R \ge r \ge r_*$ and for any $A$-harmonic function $u$ on $B_R$ (i.e., a solution of $-\nabla \cdot A \nabla u = 0$ in $B_R$) one has
\begin{equation}\nonumber
 \fint_{B_r} |\nabla u|^2 \le C(d,\lambda) \fint_{B_R} |\nabla u|^2.
\end{equation}
Moreover, assuming that the ensemble on the coefficient fields satisfies a coarsened version of the Logarithmic Sobolev Inequality, they showed that the minimal radius $r_*$ has stretched exponential moments 
\begin{equation}\nonumber
 \en{ \exp \p{ \tfrac{1}{C} r_*^{d(1-\beta)} } } \le C,
\end{equation}
where $0 \leq \beta < 1$ appearing in the exponent is related to the coarsening rate in the Logarithmic Sobolev Inequality. Observe that in the case $\beta = 0$, i.e., the case when we consider the Logarithmic Sobolev Inequality without coarsening, this is the optimal Gaussian bound.

Recently, reviving the parabolic approach used in the discrete setting \cite{GNO2short}, which has a benefit of conveniently disintegrating contributions to the corrector from different scales, Gloria and Otto~\cite{GloriaOttoNearOptimal} obtained a similar results assuming the coefficient fields have finite range of dependence (it is a known fact that the assumption of the finite range of dependence does not imply any of the functional inequalities we have mentioned above). %in this case the ensemble does not satisfy the functional inequality mentioned above). 
As a by-product, assuming finite range of dependence Gloria and Otto get the estimates for the minimal radius $r_*$ with optimal stochastic integrability of the form
\begin{equation}\nonumber
 \en{ \exp \p{ \tfrac{1}{C} r_*^{d(1-\epsilon)} } } < \infty,\quad \forall \epsilon > 0.
\end{equation}

Finally, on the other side of the spectrum, Fischer and Otto~\cite{FischerOtto2} combined Meyer's estimate together with sensitivity analysis to show that for strongly correlated coefficient fields (more precisely, they consider coefficient fields which are $1$-Lipschitz images of a stationary Gaussian field with correlations bounded by $|x|^{-\beta}$, where $0 < \beta \ll 1$ is coming from Meyer's estimates) it holds
\begin{equation}\nonumber
 \en{ \exp \p{ \tfrac{1}{C} r_*^\beta } } \le C. 
\end{equation}

For the sake of completeness, without discussing any details, let us also mention the work of Armstrong and Smart~\cite{armstrongsmart2014} which predates the previously mentioned works of Otto and coauthors (see also subsequent works \cite{ArmstrongKuusiMourrat,ArmstrongMourrat} for more general results) and  which contains estimates with the optimal stochastic integrability for some quantity related but different from $r_*$. 

In the present paper we will obtain deterministic estimates for the Green's function based on the minimal radii $r_*$ at different points. More precisely, fixing two points $x_0, y_0 \in \R^d$, we take as the input the coefficient field $A$ and the corresponding minimal radii $r_*(x_0), r_*(y_0)$ (here and in what follows $r_*(x_0)$ stays for the minimal radius $r_*$ of the shifted coefficient field $A(\cdot - x_0)$~), and produce estimates on the Green's function $G$ and its derivatives $\nabla_xG, \nabla_yG, \nabla_x\nabla_yG$, averaged over small scale around the points $x_0$ and $y_0$. This averaging is necessary since we do not assume any smoothness of the coefficient fields. 

Our only goal in this paper is to obtain bounds, and not to show existence (or other properties) of the Green's function. In fact, a well known counterexample of De Giorgi \cite{DeGiorgiCounterexample} shows that there are uniformly elliptic coefficient fields for which the Green's function does not exist. Nevertheless, as recently shown in~\cite{ConlonGiuntiOtto} by Conlon, Otto, and the second author, this is not a generic behavior. More precisely, in~\cite{ConlonGiuntiOtto} they show that for \emph{any} uniformly elliptic coefficient field $A$ the Green's function $G=G(A;x,y)$ exists at \emph{almost every} point $y \in \R^d$, provided the dimension $d \ge 3$. Therefore, in the case $d \ge 3$, we will assume that the Green's function $G(A;\cdot,y) \in L^1_{\textrm{loc}}(\R^d)$ exists, at least in the almost everywhere sense (i.e., for a.e. $y \in \R^d$), and focus solely on the estimates. Since in $\R^2$ the Green's function does not have to exist, but its ``gradient'' can possibly exists, using a reduction from $3D$ (where the Green's function exists) in Section~\ref{sec2d} we construct and estimate $\nabla G$. 

There are several works studying estimates on the Green's function in the context of uniformly elliptic equations with random coefficients. Using De Giorgi-Nash-Moser approach for a parabolic equation (which is naturally restricted to the scalar case), Delmotte and Deuschel~\cite{delmottedeuschel} obtained annealed estimates on the first and second gradient of the Green's function, in $L^2$ and $L^1$ in probability respectively, under mere assumption of stationarity of the ensemble (see also~\cite{MO2} for a different approach). Using different methods, Conlon, Otto, and the second author~\cite{ConlonGiuntiOtto} recently obtained similar estimates, together with other properties of the Green's function, but without the restriction to the scalar case. For a single equation and in the discrete case, assuming that the spatial correlation of the coefficient fields decays sufficiently fast to the effect that the Logarithmic Sobolev Inequality is satisfied, Marahrens and Otto~\cite{MO} upgraded the Delmotte-Deuschel bounds to any stochastic moments. Recently, for a single equation this work was extended by Gloria and Marahrens~\cite{gloriamarahrens} into the continuum setting. 

Before we state the main result, let us mention other works relating the smallness of the corrector and the properties of solutions to the heterogeneous equation. Together with Otto~\cite{BellaGiuntiOttoPCMI}, we compare the finite energy solution $u$ of 
\begin{equation}\nonumber
 - \nabla \cdot A \nabla u = \nabla \cdot g,
\end{equation}
with $g \in L^2(\R^d;\R^d)$ being supported in a unit ball around the origin, with \emph{twice} corrected solution  $u_\textrm{hom}$ of the homogenized equation 
\begin{equation}\nonumber
 - \nabla \cdot \Ah \nabla u_\textrm{hom} = \nabla \cdot \tilde g.
\end{equation}
Here by twice corrected we mean that first the right-hand side $g$ from the heterogeneous  equation is replaced by $\tilde g = g(\textrm{Id} + \nabla \phi)$ in the constant-coefficient equation, and second, we compare $u$ with $(\textrm{Id} + \phi_i \partial_i)v$ at the level of gradients. Using duality argument together with a compactness lemma, this gives an estimate of the difference between $\nabla_x \nabla_y G(x,y)$ and $\partial_i \partial_j G_\textrm{hom} (e_i + \nabla \phi_i(x)) \otimes (e_j + \nabla \phi_j(y))$ (averaged over small balls both in $x$ and $y$). In order to get such estimates, it is not enough to assume that the corrector is at most linear with small slope (as in~\eqref{defrstar}), but rather we need to assume that for some $\beta \in (0,1)$, it grows in the $L^2$-sense at most like $|x|^{1-\beta}$:
\begin{equation}\label{growthbeta}
 \frac{1}{R^2} \fint_{B_R} \biggl|(\phi,\sigma) - \fint_{B_R} (\phi,\sigma)\biggr|^2 \le C R^{-2\beta}, \quad \forall R \ge r_*.
\end{equation}
Hence, in comparison with the present work, in~\cite{BellaGiuntiOttoPCMI} we get a stronger statement (since we estimate the difference between the heterogeneous Green's function and corrected constant-coefficient Green's function while in the present paper we only control the heterogeneous Green's function alone), at the expense of stronger assumption on the smallness of the corrector and more involved proof. More precisely, here we show that the second mixed derivative of the Green's function $\nabla_x \nabla_y G$ behaves like $C|x-y|^{-d}$ (clearly this estimate is sharp in scaling since it agrees with the behavior of the constant-coefficient Green's function), while in \cite{BellaGiuntiOttoPCMI} we show that the homogenization error, i.e., the difference between $\nabla_x \nabla_y G$ and twice corrected mixed second derivative of the constant-coefficient Green's function, is estimated by $C|x-y|^{-(d+\beta)}$ - that means we gain a factor of $|x-y|^{-\beta}$, where the exponent $\beta \in (0,1)$ is the one appearing in \eqref{growthbeta}. 

Last, let us mention the work of Otto and the authors~\cite{BellaGiuntiOtto2nd}, where we push farther the results of~\cite{BellaGiuntiOttoPCMI} using higher order correctors. The second and  higher order correctors were introduced into the stochastic homogenization setup by Fischer and Otto~\cite{FischerOtto}, in order to extend the $C^{1,\alpha}$ regularity estimates on large scales~\cite{GNO4} to $C^{2,\alpha}$ estimates and $C^{k,\alpha}$ estimates respectively. In~\cite{BellaGiuntiOtto2nd}, under the assumption of smallness of the corrector, we obtain two results about $A$-harmonic function in exterior domains: first, for any integer $k$ we construct a finite dimensional space of functions such that the distance between any $A$-harmonic function in the exterior domain and this space is bounded by $C|x-y|^{-(d+k)}$ (this statement can be seen as an analogue of Liouville statements for finite energy solutions in the exterior domain). Second, assuming smallness of both the first and second order augmented correctors (i.e., also including the second order vector potential, which we had to introduce), compared with~\cite{BellaGiuntiOttoPCMI} we improve by $1$ the exponent in the estimate between the solution of the heterogeneous equation in the exterior domain and of some corrected solution of the constant-coefficient equation.

The paper is organized as follows: In the next section we will state our assumptions together with the main result, Theorem \ref{thm1}, and its corollaries, Corollary \ref{cor1}, Corollary 2,  and Corollary 3. In Section~\ref{sec3d} we prove Theorem~\ref{thm1} and in Section~\ref{sec2d} we give the argument for Corollary \ref{cor1}, which is the only corollary which does not immediately follow from the theorem.

\medskip

{\bf Notation.} Throughout the article, we denote by $C$ a positive generic constant which is allowed to depend on the dimension $d$ and the ellipticity contrast $\lambda$, and which may be different from line to line of the same estimate. By $\lesssim$ we will mean $\le C$. %Next,  $B_r(x)$ will denote a usual Euclidean ball of radius $r$ and centered at $x$. 
Finally, the integrals without specified domain of integration are meant as integrals over the whole space $\R^d$. 

\section{The main result}

We fix a coefficient field $A \in L^\infty(\R^d;\R^{d\times d})$, which we assume to be  uniformly elliptic in the sense that
\begin{equation}\label{coer}
\begin{aligned}
 \int_{\R^d} \nabla \varphi \cdot A(x) \nabla \varphi \dx
&\ge \lambda \int_{\R^d} \a{ \nabla \varphi }^2,& \quad &\forall \varphi \in \mathcal{C}^{\infty}_c(\R^d),\\
 |A(x)\xi| 
&\le |\xi|,& \quad &\forall \textrm{a.e. }x \in \R^d, \forall \xi \in \R^d,
\end{aligned}
\end{equation}
where $\lambda \in (0,1)$ is fixed throughout the paper. Then we have the following result:

\begin{theorem}\label{thm1}
 Let $d \ge 3$, let $A$ be a uniformly elliptic coefficient field on $\R^d$ in the sense of \eqref{coer}, and let $x_0,y_0 \in \R^d$ with $\a{x_0 - y_0} \ge 10$. For a point $x \in \R^d$, let $r_*(x)=r_*(A, x)$ denote a radius such that for $R \ge r \ge r_*(x)$ and any $A$-harmonic  function $u$ in $B_R(x)$ we have
 \begin{equation}\label{E37}
 \fint_{B_r(x)} \a{ \nabla u }^2 \le C(d,\lambda) \fint_{B_R(x)} \a{ \nabla u}^2.
\end{equation}
Let $G = G(A;x,y)$ be the Green's function defined through 
 \begin{equation}\nonumber
  -\nabla_x \cdot A \nabla_x G(A;\cdot,y) = \delta(\cdot-y),
 \end{equation}
assuming it exists for a.e. $y \in \R^d$. Then we have
\begin{align}\label{E47}
 \int_{B_1(x_0)} \int_{B_1(y_0)} \a{ \nabla_x \nabla_y G(A; x,y)}^2 \dx \dy &\le C(d,\lambda) \p{ \frac{r_*(x_0)r_*'(y_0)}{|x_0-y_0|^2}}^{d},
\\ \label{E52}
\int_{B_1(x_0)} \int_{B_1(y_0)} \a{ \nabla_y G(A; x,y)}^2 \dx \dy &\le C(d,\lambda) |x_0-y_0|^2 \p{ \frac{r_*(x_0)r_*'(y_0)}{|x_0-y_0|^2}}^{d},
\\ \label{E53}
\int_{B_1(x_0)} \int_{B_1(y_0)} \a{ \nabla_x G(A; x,y)}^2 \dx \dy &\le C(d,\lambda) |x_0-y_0|^2 \p{ \frac{r_*'(x_0)r_*(y_0)}{|x_0-y_0|^2}}^{d},
\\ \label{E54}
\int_{B_1(x_0)} \int_{B_1(y_0)} \a{ G(A; x,y)}^2 \dx \dy &\le C(d,\lambda) |x_0-y_0|^4 \p{ \frac{r_*(x_0)r_*'(y_0) + r_*'(x_0)r_*(y_0)}{|x_0-y_0|^2}}^{d}.
\end{align}
where $r_*'(y)=r_*(A^t, y)$ denotes the minimal radius for the adjoint coefficient field $A^t$ at a point $y$. 
\end{theorem}

Though the Green's function does not have to exist in $2D$, with the help of the Green's function in $3D$ we can at least define and estimate ``its gradient \& second mixed derivatives'':

\begin{corollary}\label{cor1}
Let $d=2$. Let $A$ be a uniformly elliptic coefficient field on $\R^2$ in the sense of \eqref{coer}, such that for its extension into $\R^3$ of the form
\begin{equation}\label{2d1}
 \bar A(x,x_3) := \begin{pmatrix} A(x) & 0\ \\ 0 & 1\ \end{pmatrix}
\end{equation}
there exists two points $\bar X, \bar Y \in \R^3$ so that the minimal radii for $\bar A^t$ and $\bar A$ at those points are finite, respectively (i.e., $r_*(\bar A^t,\bar X) < \infty$, $r_*(\bar A,\bar Y) < \infty$). 

Then for a.e. $y \in \R^2$ there exists a function on $\R^2$, which we denote $\nabla G(A;\cdot,y)$, so that it satisfies in a weak sense 
\begin{equation}\nonumber
 -\nabla_x \cdot A \nabla G(A;\cdot ,y) = \delta(\cdot-y).
\end{equation}
Moreover, given $x_0,y_0 \in \R^2$ with $\a{x_0 - y_0} \ge 10$, we have estimates for $\nabla G$ as well as for $\nabla_y \nabla G$:
\begin{align}
\int_{B_1(x_0)}\int_{B_1(y_0)} |\nabla_y\nabla{G}(A; x,y)|^2 \dx \dy&\le C(\lambda)  \frac{\p{r_*(A,x_0)r_*(A^t,y_0)}^2}{|x_0-y_0|^4},\label{2d5b}\\
\int_{B_1(x_0)}\int_{B_1(y_0)} |\nabla{G}(A; x,y)|^2 \dx \dy &\le C(\lambda)  \frac{\p{r_*(\bar A^t,(x_0,0))r_*(\bar A,(y_0,0))}^2}{|x_0-y_0|^2}.\label{2d5}
% \int_{B_1(x_0)}\int_{B_1(y_0)} |\nabla_y{G}(A;x,y)|^2 \dx \dy &\le C(\lambda) \frac{\p{r_*(x_0)r_*'(y_0)}^2}{R_0^2}.\label{2d6}
\end{align}
\end{corollary}

\begin{remark}
 Assuming that the coefficient field $A$ in the statement of Corollary~\ref{cor1} is chosen at random with respect to a stationary and ergodic ensemble, by the standard ergodic argument (see, e.g., \cite{GNO4}), applied in $3D$ to the ensemble obtained as a push-forward of the $2D$ ensemble through~\eqref{2d1}, the assumption on the finiteness of the minimal radii is almost surely satisfied (say with $\bar X=\bar Y=(0,0,0)$). 
\end{remark}

\begin{remark}\label{rmk1}
 It is clear from the proof of Theorem~\ref{thm1} that all the above estimates, i.e. \eqref{E47}-\eqref{2d5}, are true also if the domains of integration $B_1(x_0)$ and $B_1(y_0)$ are replaced by larger balls with the corresponding radii $r_*$.  Moreover, the radii of the balls could be even larger than the minimal radii (as long as these new radii are not larger than one third of a distance between centers of those balls), in which case we need to replace the minimal radii on the right-hand sides of those estimates with the actual radii of the balls. 
\end{remark}

\begin{remark}
 The appearance of different minimal radii in \eqref{2d5b} and \eqref{2d5} (in \eqref{2d5b} the minimal radii are related to the equation in $2D$, while in \eqref{2d5} they are the minimal radii for the equation in $3D$) is not a typo. The reason is that while \eqref{2d5b} is proved directly in $2D$, the proof of \eqref{2d5} passes through $3D$ - hence the need to consider the minimal radii in $3D$. In view of the relation $r_*(A,x) \le r_*(\bar A,\bar x)$, which easily follows from the fact that any $A$-harmonic function in $B_R \subset \R^2$ can be trivially extended to an $\bar A$-harmonic function in $\bar B_R \subset \R^3$, the estimate \eqref{2d5} seems to be less optimal.
\end{remark}

For notational convenience we state the result for a single equation. Since in the proof of Theorem~\ref{thm1} we do not use any \emph{scalar} methods (like for example De Giorgi-Nash-Moser iteration), the result holds also in the case of elliptic systems - for that one just considers that $u$ has values in some finite-dimensional Hilbert space. Naturally, in that case all the constants will depend on the dimension of this Hilbert space. 

Using the Gaussian bounds on $r_*$ for the case of coefficient fields with finite range of dependence, which were obtained recently in~\cite{GloriaOttoNearOptimal}, Theorem~\ref{thm1} implies the following bounds:

\begin{corollary}
Suppose $\en{\cdot}$ is an ensemble of $\lambda$-uniformly elliptic coefficient fields
which is stationary and of unity range of dependence, and let $d \ge 2$. Then there exist $C(d,\lambda)$ such that for every two points $x_0,y_0 \in \R^d$, $|x_0 - y_0| \ge 10$, and every $\epsilon > 0$ we have 
\begin{align*}
 \en{ \exp \biggl( \biggl( C|x_0 - y_0|^{2d} \int_{B_1(x_0)} \int_{B_1(y_0)} |\nabla_x \nabla_y G( A ; x,y)|^2 \dx \dy \biggr)^{d(1-\epsilon)} \biggr) } &< \infty,
\\
 \en{ \exp \biggl( \biggl( C|x_0 - y_0|^{2d-2} \int_{B_1(x_0)} \int_{B_1(y_0)} |(\nabla_x,\nabla_y) G( A ; x,y)|^2 \dx \dy \biggr)^{d(1-\epsilon)} \biggr) } &< \infty,
\end{align*}
and in $d \ge 3$ also
\begin{equation}\nonumber
 \en{ \exp \biggl( \biggl( C|x_0 - y_0|^{2d-4} \int_{B_1(x_0)} \int_{B_1(y_0)} |G( A ; x,y)|^2 \dx \dy \biggr)^{d(1-\epsilon)} \biggr) } < \infty.
\end{equation}
\end{corollary}

In the case of coefficient fields with stronger correlations we can use the result from~\cite{GNO4}:

\begin{corollary}
Suppose $d \ge 2$, and that the ensemble $\en{\cdot}$ is stationary and satisfies a logarithmic  Sobolev
inequality of the following  type:  There  exists a partition $\{ D \}$ of $\R^d$ not too coarse in the sense that for some $0 \le \beta < 1$ it satisfies 
\begin{equation*}
 \textrm{diam} (D) \le (\textrm{dist}(D) + 1)^\beta \le C(d) \textrm{diam}(D).
\end{equation*}
Moreover, let us assume that there is $0 < \rho \le 1$ such that for all random variables $F$
\begin{equation*}
 \en{ F^2 \log F^2 } - \en{F^2} \log \en{F^2} \le \frac{1}{\rho} \biggl< \left\| \frac{\partial F}{\partial A}\right\|^2 \biggr>,
\end{equation*}
where the carr\'e-du-champ of the Malliavin derivative is defined as 
\begin{equation*}
\biggl\| \frac{\partial F}{\partial A}\biggr\|^2 := \sum_D \biggl( \int_D \biggl|\frac{\partial F}{\partial A}\biggr| ^2 \biggr).
\end{equation*}
Then  there  exists  a  constant $0 < C < \infty$, depending only on $d, \lambda, \rho, \beta$, such that 
\begin{align*}
 \en{ \exp \biggl( \biggl( C|x_0 - y_0|^{2d} \int_{B_1(x_0)} \int_{B_1(y_0)} |\nabla_x \nabla_y G( A ; x,y)|^2 \dx \dy \biggr)^{d(1-\beta)} \biggr) } &< \infty,
\\
 \en{ \exp \biggl( \biggl( C|x_0 - y_0|^{2d-2} \int_{B_1(x_0)} \int_{B_1(y_0)} |(\nabla_x,\nabla_y) G( A ; x,y)|^2 \dx \dy \biggr)^{d(1-\beta)} \biggr) } &< \infty,
\end{align*}
and in $d \ge 3$ also
\begin{equation*}
 \en{ \exp \biggl( \biggl( C|x_0 - y_0|^{2d-4} \int_{B_1(x_0)} \int_{B_1(y_0)} |G( A ; x,y)|^2 \dx \dy \biggr)^{d(1-\beta)} \biggr) } < \infty.
\end{equation*}

\end{corollary}

\section{Proof of Theorem \ref{thm1}}\label{sec3d}

The proof is inspired by a duality argument of Avellaneda and Lin \cite[Theorem 13]{AL1}, which they used to obtain Green's function estimates in the periodic homogenization. After stating and proving two auxiliary lemmas, we first prove the estimate on the second mixed derivative \eqref{E47}. Then, \eqref{E52} will follow from \eqref{E47} using Poincar\'e inequality and one additional estimate. Next we observe that \eqref{E53} can be obtained from \eqref{E52} by replacing the role of $x$ and $y$, which can be done by considering the adjoint $A^t$ instead of $A$. Finally, \eqref{E54} will follow from \eqref{E53} in a similar way as \eqref{E52} follows from \eqref{E47}. 

We thus start with the following two auxiliary lemmas. The first one is very standard:
%Before we start with the actual proof, let us state and proof two simple auxiliary lemmas. 

\begin{lemma}[Caccioppoli inequality]\label{lmcac}
 Let $\rho > 0$, $\delta > 0$, and let $u$ be a solution of a uniformly elliptic equation $-\nabla \cdot A\nabla u = 0$ in $B_{(1+\delta)\rho}$. Then 
 \begin{equation}\label{E99}
  \int_{B_\rho} \a{\nabla u}^2 \le \frac{C(d)}{\lambda\rho^2\delta^2} \int_{B_{(1+\delta)\rho}} \a{u - c}^2
 \end{equation}
 for any $c \in \R$. 
\end{lemma}

\begin{proof}
By considering $u-c$ instead of $u$, it is enough to show estimate (\ref{E99}) with $c = 0$. We test the equation for $u$ with $\eta^2 u$, where $\eta$ is a smooth cut-off function for $B_\rho$ in $B_{(1+\delta)\rho}$ with $\a{\nabla \eta} \lesssim (\delta\rho)^{-1}$, use \eqref{coer} and Young's inequality to get
\begin{equation}\nonumber
 \int_{\R^d} \a{ \nabla (\eta u) }^2 \le \frac{C(d)}{\lambda} \int \a{ \nabla \eta}^2 u^2. 
\end{equation}
Since $\a{\nabla \eta} \le \frac{C}{\rho \delta}$, \eqref{E99} immediately follows.
\end{proof}

\begin{lemma}\label{lm1}
 Let $R_0 \ge r_*(0)$, and let $u$ be an $A$-harmonic function in $B_{R_0}$.
 Then we have
 \begin{equation}\label{E130}
  \fint_{B_{r_*(0)}} \a{u}^2 \le C(d,\lambda) \fint_{B_{R_0}} \a{u}^2.
 \end{equation}
\end{lemma}

\begin{proof}
 Throughout the proof we write $r_*$ instead of $r_*(0)$. We assume that $2r_* < R_0$, since otherwise \eqref{E130} is trivial. For $r \in [r_*,R_0]$ we denote $u_r := \fint_{B_r} u$. We have
 \begin{align*}
  \fint_{B_{r_*}} \a{u - u_{r_*}}^2 &\overset{\textrm{Poincar\'e}}{\lesssim} r_*^2 \fint_{B_{r_*}} \a{ \nabla u }^2 \overset{\eqref{E37}}{\lesssim} r_*^2 \fint_{B_{{R_0/2}}} \a{ \nabla u}^2 
\\ &\uoset{\textrm{Poincar\'e}}{\eqref{E99}}{\lesssim}
\p{ \frac{r_*}{R_0} }^{2}  \fint_{B_{R_0}} \a{u}^2 \le \fint_{B_{R_0}} \a{u}^2.
 \end{align*}
Hence, to prove \eqref{E37} it is enough to show 
\begin{equation}\label{E198}
 \a{ u_{r_*} }^2 = \a{ \fint_{B_{r_*}} u }^2 \lesssim \fint_{B_{R_0}} \a{ u }^2.
\end{equation}
To prove it, we use the following estimate
\begin{equation}\label{E199}
 \a{ u_r - u_{2r} } \lesssim r \p{ \fint_{B_{2r}} \a{ \nabla u }^2 }^\oh,
\end{equation}
which in fact holds for any function $u \in W^{1,2}(B_{2r})$. 

We first argue how to obtain (\ref{E198}) the proof thanks to estimate (\ref{E199}): Let $n \ge 0$ be the largest integer that satisfies $2^n r_* \le R_0/2$; using \eqref{E199} multiple times we get
\begin{align*}
 \a{ u_{r_*} - u_{2^n r_*} } &\uoset{\eqref{E37}}{}{\le} \sum_{k=0}^{n-1} \a{ u_{2^k r_*} - u_{2^{k+1} r_*} } \overset{\eqref{E199}} \lesssim \sum_{k=0}^{n-1} 2^k r_* \p{ \fint_{B_{2^{k+1}r_*}} \a{ \nabla u }^2 }^\oh 
\\
&\overset{\eqref{E37}}\lesssim \p{ \fint_{B_{R_0/2}} \a{ \nabla u }^2 }^\oh \sum_{k=0}^{n-1} 2^k r_* \overset{\eqref{E99}} \lesssim R_0 \p{ \frac{1}{R_0^2} \fint_{B_{R_0}} \a{  u }^2 }^\oh = \p{ \fint_{B_{R_0}} \a{  u }^2 }^\oh.
\end{align*}
Using Jensen's inequality and the fact that $R_0 \le 2^{n+2}r_*$ we get
\begin{equation}\nonumber
 \a{ u_{2^n r_*} } = \a{ \fint_{B_{2^n r_*}} u } \le \p{ \fint_{B_{2^n r_*}} \a{u}^2 }^\oh \lesssim \p{ \fint_{B_{R_0}} \a{  u }^2 }^\oh.
\end{equation}
Combination of the two previous estimates then gives \eqref{E198}. 

It remains to prove \eqref{E199}. Using Jensen's and Poincar\'e's inequalities we get
\begin{align*}
 \a{ u_r - u_{2r} } &= \a{ \fint_{B_r} (u-u_r) - (u-u_{2r}) } \lesssim \fint_{B_r} \a{ u - u_r} + \fint_{B_{2r}} \a{ u - u_{2r}} 
\\ 
&\lesssim \p{ \fint_{B_r} \a{ u - u_r}^2}^\oh + \p{\fint_{B_{2r}} \a{ u - u_{2r}}^2 }^\oh \lesssim r \p{ \fint_{B_{2r}} \a{ \nabla u }^2}^\oh.
\end{align*}
\end{proof}

\subsection{Proof of \eqref{E47}}

We denote $R_0 := |x_0 - y_0|/3$. We split the proof of \eqref{E47} into 4 steps. In the first step we show that 
\begin{equation}\label{E174}
 \int_{B_1(y_0)} \a{ \F\p{\nabla_x \nabla_y G(\cdot,y)}}^2 \dy \lesssim \p{ \frac{r_*(x_0)r_*'(y_0)}{R_0^2}}^{d}
\end{equation}
for any $\rho \in [1,2]$ and any functional $F_\rho$ on $L^2(B_\rho(x_0))$ which satisfies
\begin{equation}\label{E62}
 \a{ \F(\nabla v) }^2 \le \int_{B_\rho(x_0)} \a{ \nabla v }^2,
\end{equation}
for any $v \in W^{1,2}(B_\rho(x_0))$.
In the second step, using Neumann eigenfunctions on a ball, we define a family of functionals $F_k$ satisfying \eqref{E62}, which will play the role of Fourier coefficients. Using these $F_k$ we then estimate $\int \a{v}^2$ with a sum of $N$ terms $\a{F_k(v)}^2$ plus a residuum in the form $\frac{1}{\lambda_N} \int \a{\nabla v}^2$. Here, $\lambda_N$ denotes the $N$-th Neumann eigenvalue of Laplacian on a ball. Using \eqref{E62} together with the second step we get an estimate on $\int \a{ \nabla_x \nabla_y G}^2$ in terms of a (good) term and a small prefactor times $\int \a{ \nabla_x \nabla_y G}^2$, integrated over a slightly larger ball. In the last step, we apply iteratively this estimate.  

\medskip\noindent{\bf Step 1.} Proof of \eqref{E174} (inspired by the duality argument of Avellaneda and Lin \cite{AL1}). 

Let $f \in L^2(B_{R_0}(y_0);\R^d)$, and let $u$ be the finite energy solution of 
\begin{equation}\nonumber
 -\nabla \cdot A \nabla u = -\nabla \cdot f
\end{equation}
in $\R^d$, for which holds the energy estimate
\begin{equation}\label{energy1}
 \int_{\R^d} \a{\nabla u}^2 \lesssim \int_{\R^d} \a{f}^2.
\end{equation}
Then on the one hand, the Green's function representation formula yields
\begin{equation}\label{greenrep}
 \nabla u(x) = \int_{B_{R_0}(y_0)} \nabla_x \nabla_y G(x,y) f(y) \dy.
\end{equation}
If $r_*(x_0) \le R_0$ (w.\ l.\ o.\ g.\ we assume $r_*(x_0) \ge \rho$), we use this in \eqref{E37} to get 
\begin{multline}\nonumber
 \a{ \F(\nabla u) }^2 \le \int_{B_\rho(x_0)} \a{ \nabla u}^2 \dx \le \int_{B_{r_*(x_0)}(x_0)} \a{\nabla u}^2 \lesssim \p{ \frac{r_*(x_0)}{R_0}}^d \int_{B_{R_0}(x_0)} \a{\nabla u}^2 
\\
\overset{\eqref{energy1}}{\lesssim} \p{ \frac{r_*(x_0)}{R_0}}^d \int_{\R^d} \a{f}^2. 
\end{multline}
If $r_*(x_0) \ge R_0$, we simply have
\begin{equation}\nonumber
 \a{ \F(\nabla u) }^2 \le \int_{B_\rho(x_0)} \a{ \nabla u}^2 \dx \overset{\eqref{energy1}}\lesssim \int_{\R^d} \a{f}^2 \dx \le \p{ \frac{r_*(x_0)}{R_0}}^d \int_{\R^d} \a{f}^2. 
\end{equation}

Since $\F$ is linear, using \eqref{greenrep} we have
\begin{equation}\nonumber
 \F(\nabla u) = \int_{B_{R_0}(y_0)} \F\p{ \nabla_x \nabla_y G(\cdot,y)} f(y) \dy,
\end{equation}
where the dot means that $\F$ acts on the first variable. The previous relations then give
\begin{equation}\nonumber
 \a{ \int_{B_{R_0}(y_0)} \F\p{ \nabla_x \nabla_y G(\cdot,y)} f(y) \dy }^2 \lesssim \p{ \frac{r_*(x_0)}{R_0}}^d \int_{B_{R_0}(x_0)} \a{f}^2. 
\end{equation}
Using definition of the norm $L^2(B_{R_0}(y_0))$ by duality we get
\begin{equation}\label{E98}
 \int_{B_{R_0}(y_0)} \a{ \F\p{ \nabla_x \nabla_y G(\cdot,y)} }^2 \dy 
\lesssim \p{ \frac{r_*(x_0)}{R_0}}^{d}.
\end{equation}

Let $r_*'(y_0)$ play the same role as $r_*(x_0)$ but for the adjoint equation. In the case $r_*'(y_0) \le R_0$, since $y \mapsto \F(\nabla_x G(\cdot,y))$ solves the adjoint equation $-\nabla \cdot A^t \nabla\F(\nabla_x G(\cdot,y))=0$ in $B_{R_0}(y_0)$, an analogue of \eqref{E37} implies
\begin{multline}\label{E88}
 \int_{B_1(y_0)} \a{ \F(\nabla_x \nabla_y G(\cdot,y)) }^2 \dy \le \int_{B_{r_*'(y_0)}(y_0)} \a{ \F(\nabla_x \nabla_y G(\cdot,y)) }^2 \dy 
\\
\lesssim \p{ \frac{r_*'(y_0)}{R_0}}^{d} \int_{B_{R_0}(y_0)} \a{ \F(\nabla_x \nabla_y G(\cdot,y)) }^2 \dy \overset{\eqref{E98}}{\lesssim} \p{ \frac{r_*(x_0)r_*'(y_0)}{R_0^2}}^{d}.
\end{multline}
If $r_*'(y_0) \ge R_0$, we get the same conclusion for free:
\begin{align*}
 \int_{B_1(y_0)} \a{ \F(\nabla_x \nabla_y G(\cdot,y)) }^2 \dy 
&\uoset{\eqref{E98}}{}{\le} \int_{B_{R_0}(y_0)} \a{ \F(\nabla_x \nabla_y G(\cdot,y)) }^2 \dy 
\\
&\overset{\eqref{E98}}\lesssim \p{ \frac{r_*(x_0)}{R_0}}^{d} \le \p{ \frac{r_*(x_0)r_*'(y_0)}{R_0^2}}^{d}.
\end{align*}

\medskip\noindent{\bf Step 2.} Let $\rho \in [1,2)$ and $\delta > 0$ be fixed such that $(1+\delta)\rho \le 2$. For $n \ge 0$, let $F_n$ denote the functional on $L^2(B_{(1+\delta)\rho})$, defined as an inner product of the $n$-th eigenfunction of Neumann Laplacian with $v$, and let $\lambda_n$ be the associated eigenvalue. By writing any $v \in W^{1,2}(B_{(1+\delta)\rho})$ in terms of Neumann eigenfunctions (which can be done since they form an orthonormal basis in $L^2$) we get
\begin{equation}\label{E110}
\begin{aligned}
  \int_{B_{(1+\delta)\rho}} \a{ v }^2 &= \sum_{k=0}^\infty \a{ F_k(v) }^2 = 
\sum_{k=0}^{N-1} \a{ F_k(v) }^2 + \sum_{k=N}^\infty \frac{1}{\lambda_k} \a{ F_k(\nabla v) }^2 \\ 
&\le \sum_{k=0}^{N-1} \a{ F_k(\nabla v) }^2 + \frac{1}{\lambda_N} \int_{B_{(1+\delta)\rho}} \a{ \nabla v }^2.
\end{aligned}
\end{equation}

\medskip\noindent{\bf Step 3.} Combination of Step 1 and Step 2 (applied to $\nabla_y G(\cdot,y)$) and use of \eqref{E88} yields
\begin{equation}
\begin{aligned}\label{43}
 &\int_{B_1(y_0)} \int_{B_\rho(x_0)} \a{ \nabla_x \nabla_y G(x,y) }^2 \dx \dy \overset{\eqref{E99}}\lesssim 
 \frac{1}{\delta^2} \int_{B_1(y_0)} \int_{B_{(1+\delta)\rho}(x_0)} \a{ \nabla_y G(x,y) }^2 \dx \dy \le
 \\
 &\overset{\eqref{E110}}\lesssim \frac{1}{\delta^2} \p{ \sum_{k=0}^{N-1} \int_{B_1(y_0)} \a{ F_k(\nabla_x \nabla_y G(\cdot,y))}^2 + \frac{1}{\lambda_N} \int_{B_1(y_0)} \int_{B_{(1+\delta)\rho}(x_0)} \a{ \nabla_x \nabla_y G(x,y) }^2 \dx \dy }
 \\
 &\overset{\eqref{E88}}\lesssim \frac{1}{\delta^2} \p{ N \p{ \frac{r_*(x_0)r_*'(y_0)}{R_0^2}}^{d} + \frac{1}{\lambda_N} \int_{B_1(y_0)} \int_{B_{(1+\delta)\rho}(x_0)} \a{ \nabla_x \nabla_y G(x,y) }^2 \dx \dy }.
\end{aligned}
\end{equation}

\medskip\noindent{\bf Step 4.} For a given sequence $\delta_k > 0$ such that $\rho \Pi_{k=1}^\infty (1+\delta_k) \le 2$ we consider the following iteration procedure. Let $\rho_0 := 1$, and for $k\ge 1$ set $\rho_k := (1+\delta_k)\rho_{k-1}$. We denote
\begin{equation}\nonumber
M_k := \p{ \frac{r_*(x_0)r_*'(y_0)}{R_0^2}}^{-d} \int_{B_1(y_0)} \int_{B_{\rho_k}(x_0)} \a{ \nabla_x \nabla_y G(x,y)) }^2 \dx \dy. 
\end{equation}
For any $N_k \ge 1$, estimate~\eqref{43} in Step 3 yields
\begin{equation}\label{E121}
 M_k \le \frac{C}{\delta_k^2} N_k + \frac{C}{\delta_k^2} \frac{1}{\lambda_N} M_{k+1},
\end{equation}
where the values of $\delta_k$ and $N_k$ are at our disposal. We choose $\delta_k := (2k)^{-2}$ and $N_k := \alpha k^{2d} 2^d$. Since $\Pi_{k=1}^\infty (1+\frac{1}{4k^2}) \sim 1.46 \le 2$, for this choice of $\delta_k$ for all $k\ge 1$ we have $\rho_k \in [1,2]$. Using lower bound on the Neumann eigenvalues for the ball in the form $\lambda_k \ge Ck^{\frac{2}{d}}$ (in the case of a cube one can use trigonometric functions to explicitly write down the formula for eigenfunctions and eigenvalues; for the ball, one uses the monotonicity of the eigenvalues with respect to the domain, which follows from the variational formulation of the eigenvalues), we can find large enough $\alpha$ such that the prefactor in front of $M_{k+1}$ above satisfies
\begin{equation}\nonumber
 \frac{C}{\delta_k^2} \frac{1}{\lambda_N} \le C' k^4 (\alpha k^{2d} 2^d)^{-\frac{2}{d}} = \frac{C'}{\alpha^{\frac{2}{d}}} \frac{1}{4} \le \frac{1}{4}.
\end{equation}
For this choice \eqref{E121} turns into
\begin{equation}\nonumber
 M_k \le C\alpha k^4 k^{2d} 2^d + \frac{1}{4} M_{k+1}.
\end{equation}
Iterating this we get 
\begin{equation}\nonumber
 M_1 \le C\alpha \sum_{k=1}^K 4^{-k} k^4 k^{2d} 2^d + \p{\frac{1}{4}}^{K} M_{K+1}.
\end{equation}
Assuming we have $\sup_k M_k < \infty$, we send $K \to \infty$ to get
\begin{equation}\nonumber
 M_1 \le C\alpha2^d \sum_{k=1}^\infty 4^{-k} k^{4+2d}.
\end{equation}
Since the sum on the right-hand side is summable, we get that $M_1 \lesssim 1$. 

It remains to justify the assumption $\sup_k M_k < \infty$. For any $\Lambda \ge 1$, let $\chi_\Lambda(y)$ be the characteristic function of the set $\left\{ y \in B_1(y_0) : \int_{B_2(x_0)} |\nabla_x \nabla_y G(x,y)|^2 \le \Lambda \right\}$. Using the previous arguments, applied to $\nabla_x\nabla_y G(x,y) \chi_\Lambda(y)$, we get that 
\begin{equation}\nonumber
\p{ \frac{r_*(x_0)r_*'(y_0)}{R_0^2}}^{-d} \int_{B_1(y_0)} \biggl( \int_{B_{1}(x_0)} \a{ \nabla_x \nabla_y G(x,y)) }^2 \dx\biggr) \chi_\Lambda(y) \dy \le C,
\end{equation}
where the right-hand side does not depend on $\Lambda$. Now we send $\Lambda \to \infty$, and get 
\begin{equation}\nonumber
M_1 = \p{ \frac{r_*(x_0)r_*'(y_0)}{R_0^2}}^{-d} \int_{B_1(y_0)} \int_{B_{1}(x_0)} \a{ \nabla_x \nabla_y G(x,y)) }^2 \dx \dy \le C
\end{equation}
by the Monotone Convergence Theorem. This completes the proof of \eqref{E47}.

% \end{proof}

\subsection{Proof of \eqref{E52}}

We first observe that using Poincar\'e's inequality we can control the difference between $\nabla_y G$ and its averages over $B_1(x_0)$ by the $L^2$-norm of $\nabla_x \nabla_y G$, which we already control by \eqref{E47}. Hence, to obtain \eqref{E52} it is enough to estimate averages $\fint_{B_1(x_0)} \nabla_y G(x,y) \dx$. Such estimate will follow from an analogue of \eqref{E174} applied to one particular functional $F$. Since in this setting we need to  work with $\int \a{u}^2$ and not with previously used $\int \a{\nabla u}^2$, we will need to use Lemma \ref{lm1}.

\medskip\noindent{\bf Step 1.} 
By Poincar\'e inequality in the $x$-variable we have
\begin{align}\label{E288}
 &\int_{B_1(y_0)} \p{ \int_{B_1(x_0)} \a{ \nabla_y G(x,y) - \p{ \fint_{B_1(x_0)} \nabla_y G(x',y) \ud x' }}^2 \dx} \dy
\\  \nonumber
 &\uoset{\eqref{E47}}{}{\lesssim} \int_{B_1(y_0)} \int_{B_1(x_0)} \a{ \nabla_x \nabla_y G(x,y)}^2 \dx \dy 
 \\ \nonumber
 &\overset{\eqref{E47}}\lesssim \p{ \frac{r_*(x_0)r_*'(y_0)}{R_0^2}}^{d}.
\end{align}
By the triangle inequality we have
\begin{align*}
 &\int_{B_1(y_0)} \int_{B_1(x_0)} \a{ \nabla_y G(x,y)}^2 \dx \dy 
\\ 
&\lesssim \int_{B_1(y_0)} \p{ \int_{B_1(x_0)} \a{ \nabla_y G(x,y) - \p{ \fint_{B_1(x_0)} \nabla_y G(x',y) \ud x' }}^2 \dx} \dy
\\
&\quad +\a{B_1} \int_{B_1(y_0)} \p{ \fint_{B_1(x_0)} \nabla_y G(x,y) \dx}^2 \dy,
\end{align*}
and so \eqref{E52} follows from \eqref{E288} provided we show
%Using \eqref{E268} for the functional $F(v) = \fint_{B_1(x_0)} v(x) \dx$, we get that 
\begin{align}\label{E312}
 \int_{B_1(y_0)} \p{ \fint_{B_1(x_0)} \nabla_y G(x,y) \dx}^2 \dy \lesssim \frac{\p{r_*(x_0)r_*'(y_0)}^d}{R_0^{2d-2}}.
\end{align}
% By the triangle inequality, \eqref{E288} and \eqref{E295} imply \eqref{E52}. 

\medskip\noindent{\bf Step 2.} Proof of \eqref{E312}. Similarly as for \eqref{E47}, consider arbitrary $f \in L^2(B_{R_0}(y_0);\R^d)$ and the finite energy solution $u$ of 
\begin{equation}\nonumber
 - \nabla \cdot A \nabla u = -\nabla \cdot f
\end{equation}
in $\R^d$, which satisfies the energy estimate
\begin{equation}\label{E246}
\int_{\R^d} \a{ \nabla u}^2 \lesssim \int_{\R^d} \a{f}^2. 
\end{equation}
% We point out that compared to the proof of \eqref{E47}, we got additional $R_0^2$ due to the right-hand side of the equation being $f$ and not $\nabla \cdot f$. 
%
Let $F$ be a linear functional on $L^2(B_1(x_0))$ such that $\a{F(v)}^2 \le \int_{B_1(x_0)} \a{v}^2$ for any $v \in L^2(B_1(x_0))$. 
Then, if $r_*(x_0) \leq R_0$
\begin{align*}
 \a{ F(u) }^2 &\uoset{\textrm{Sobolev}}{}{\le} \int_{B_1(x_0)} \a{ u }^2 \le \int_{B_{r_*(x_0)}(x_0)} \a{u}^2 \overset{\textrm{Lemma \ref{lm1}}}\lesssim r_*^d(x_0) \fint_{B_{R_0}(x_0)} \a { u}^2 
 \\
 &\uoset{\textrm{Sobolev}}{\textrm{Jensen}}{\le} r_*^d(x_0) \p{ \fint_{B_{R_0}(x_0)} \a { u}^{\frac{2d}{d-2}} }^{\frac{d-2}{d}} 
 \lesssim \frac{r_*^d(x_0)}{R_0^{d-2}} \p{ \int_{\R^d} \a { u}^{\frac{2d}{d-2}} }^{\frac{d-2}{d}}
 \\
  &\overset{\textrm{Sobolev}}\lesssim \frac{r_*^d(x_0)}{R_0^{d-2}} \int_{\R^d} \a{\nabla u}^2 \overset{\eqref{E246}}\lesssim \frac{r_*^d(x_0)}{R_0^{d-2}} \int_{\R^d} \a{f}^2. 
\end{align*}
If otherwise $r_*(x_0) > R_0$, we do not need anymore to appeal to Lemma \ref{lm1} and may directly bound
\begin{align*}
\a{ F(u) }^2 &\uoset{\textrm{Sobolev}}{}{\le} \int_{B_1(x_0)} \a{ u }^2 \le \int_{B_{R_0}(x_0)} \a{u}^2 \lesssim r_*^d(x_0) \fint_{B_{R_0}(x_0)} \a { u}^2 
\end{align*}
and proceed as in the previous inequality.
As before, we use linearity of $F$ and write
\begin{equation}\nonumber
 \a{ F(u) } = \a{ \int_{B_{R_0}(y_0)} F( \nabla_y G(\cdot,y) ) f(y) \dy }.
\end{equation}
Since $f \in L^2(B_{R_0}(y_0);\R^d)$ was arbitrary, combination of the two previous estimates yields
\begin{equation}\label{E260}
 \int_{B_{R_0}(y_0)} \a{ F(\nabla_y G(\cdot,y)) }^2 \dy \lesssim \frac{r_*^d(x_0)}{R_0^{d-2}}.
\end{equation}
As before, it remains to argue that by going from $\int_{B_{R_0}(y_0)}$ to $\int_{B_1(y_0)}$ we gain a factor $R_0^{-d}$. We define $v(y) := F(G(\cdot,y))$, and observe that $-\nabla A^t \nabla v = 0$ in $B_{R_0}(y_0)$, where $A^t$ denotes the adjoint coefficient field. %(in coordinates this means $A_{ij}^{\alpha\beta} = (A^t)_{ji}^{\beta\alpha}$). Let $r_0'(y_0)$ denotes the minimal diameter from which $\fint_{B_r(y_0)} \a{ \nabla v}^2$ is non-decreasing. 
% Then by Lemma \ref{lm1}, applied for $v$, we get 
% \begin{equation}
%  \int_{B_1(x_0)} \a{ v }^2  \le \int_{B_{r_*'(y_0)}} \a{v}^2 \overset{\textrm{Lemma \ref{lm1}}}\lesssim \p{\frac{r_*'(y_0)}{R_0}}^d \int_{B_{R_0}(y_0)} \a{v}^2.
% \end{equation}
Then by definition of $v$ estimate \eqref{E260} implies
\begin{equation}\label{E268}
\begin{aligned}
 \int_{B_{1}(y_0)} \a{ F(\nabla_y G(\cdot,y)) }^2 \dy &= \int_{B_{1}(y_0)} \a{ \nabla v }^2 \dy \le \int_{B_{r_*'(y_0)}(y_0)} \a{ \nabla v }^2 \dy 
 \\
 &\lesssim \p{\frac{r_*'(y_0)}{R_0}}^d \int_{B_{R_0}(y_0)} \a{ \nabla v }^2 \lesssim \frac{\p{r_*(x_0)r_*'(y_0)}^d}{R_0^{2d-2}}.
 \end{aligned}
\end{equation}
For the choice $F(v) = \fint_{B_1(x_0)} v$ \eqref{E268} is exactly \eqref{E312}. 

\subsection{Proof of \eqref{E54}}

Similarly to the proof of \eqref{E52}, we use Poincar\'e's inequality (Step 1) to show that \eqref{E54} follows from \eqref{E53} provided we control averages of $G$ (Step 2). 

\medskip\noindent{\bf Step 1.} 
By Poincar\'e's inequality in the $x$-variable we have
\begin{align}\nonumber%\label{E385}
 &\int_{B_1(y_0)} \p{ \int_{B_1(x_0)} \a{ G(x,y) - \p{ \fint_{B_1(x_0)} G(x',y) \ud x' }}^2 \dx} \dy
\\ \nonumber
 &\uoset{\eqref{E53}}{}{\lesssim} \int_{B_1(y_0)} \int_{B_1(x_0)} \a{ \nabla_x G(x,y)}^2 \dx \dy 
 \\ \nonumber
 &\overset{\eqref{E53}}\lesssim R_0^2 \p{ \frac{r_*'(x_0)r_*(y_0)}{R_0^2}}^{d}.
\end{align}
%where $\overline{\nabla_y G(\cdot,y)} = \fint_{B_1(x_0)} \nabla_y G(x,y) \dx$. 
Then by the triangle inequality we have
\begin{align} \nonumber
 &\int_{B_1(y_0)} \int_{B_1(x_0)} \a{ G(x,y)}^2 \dx \dy 
\\ \nonumber
&\lesssim \int_{B_1(y_0)} \p{ \int_{B_1(x_0)} \a{ G(x,y) - \p{ \fint_{B_1(x_0)} G(x',y) \ud x' }}^2 \dx} \dy 
\\ \nonumber
&\quad + \a{B_1} \int_{B_1(y_0)} \p{ \fint_{B_1(x_0)} G(x,y) \dx}^2 \dy,
\end{align}
and so \eqref{E54} follows provided we show
%Using \eqref{E268} for the functional $F(v) = \fint_{B_1(x_0)} v(x) \dx$, we get that 
\begin{align}\label{E402}
 \int_{B_1(y_0)} \p{ \fint_{B_1(x_0)} G(x,y) \dx}^2 \dy \lesssim \frac{\p{r_*(x_0)r_*'(y_0)}^d}{R_0^{2d-4}}.
\end{align}

\medskip\noindent{\bf Step 2.} Proof of \eqref{E402}. Similarly as for \eqref{E52}, consider arbitrary $f \in L^2(B_{R_0}(y_0))$, but this time $u$ being a finite energy solution of 
\begin{equation}\nonumber
 - \nabla \cdot A \nabla u = f
\end{equation}
in $\R^d$. In order to get the energy estimate, we test the equation with $u$ to obtain:
\begin{align*}
\lambda \int_{\R^d} \a{\nabla u}^2 &\uoset{\textrm{Jensen},d\ge 3}{}{\le} \int_{B_{R_0}(y_0)} f u \le R_0^\frac{d}{2}\p{ \int_{B_{R_0}(y_0)} \a{f}^2 }^\oh \p{ \fint_{B_{R_0}(y_0)} \a{u}^2 }^\oh 
\\
&\overset{\textrm{Jensen},d\ge 3}\le 
R_0^\frac{d}{2}\p{ \int_{B_{R_0}(y_0)} \a{f}^2 }^\oh \p{ \fint_{B_{R_0}(y_0)} \a{u}^\frac{2d}{d-2} }^\frac{d-2}{2d}
\\
&\uoset{\textrm{Jensen},d\ge 3}{}{=}
R_0 \p{ \int_{B_{R_0}(y_0)} \a{f}^2 }^\oh \p{ \int_{B_{R_0}(y_0)} \a{u}^\frac{2d}{d-2} }^\frac{d-2}{2d} 
\\
&\uoset{\textrm{Jensen},d\ge 3}{\textrm{Sobolev}}{\lesssim}
 R_0 \p{ \int_{B_{R_0}(y_0)} \a{f}^2 }^\oh \p{ \int_{\R^d} \a{\nabla u}^2 }^\oh,
\end{align*}
and so
\begin{equation}\label{E424}
 \int_{\R^d} \a{ \nabla u }^2 \lesssim R_0^2 \int_{B_{R_0}(y_0)} \a{f}^2. 
\end{equation}
We point out that compared to the proof of \eqref{E47} or \eqref{E52}, we got additional $R_0^2$ due to the right-hand side of the equation being $f$ and not $\nabla \cdot f$. 

Let $F$ be a linear functional on $L^2(B_1(x_0))$ such that $\a{F(v)}^2 \le \int_{B_1(x_0)} \a{v}^2$. 
If $r_*(x_0) \leq R_0$, then
\begin{align*}
 \a{ F(u) }^2 &\uoset{\textrm{Jensen},d\ge 3}{}{\le} \int_{B_1(x_0)} \a{ u }^2 \le \int_{B_{r_*(x_0)}(x_0)} \a{u}^2 \overset{\textrm{Lemma \ref{lm1}}}\lesssim r_*^d(x_0) \fint_{B_{R_0}(x_0)} \a { u}^2 
 \\
 &\overset{\textrm{Jensen},d\ge 3}\le r_*^d(x_0) \p{ \fint_{B_{R_0}(x_0)} \a { u}^{\frac{2d}{d-2}} }^{\frac{d-2}{d}} 
 \lesssim \frac{r_*^d(x_0)}{R_0^{d-2}} \p{ \int_{\R^d} \a { u}^{\frac{2d}{d-2}} }^{\frac{d-2}{d}}
 \\
  &\uoset{\textrm{Jensen},d\ge 3}{\textrm{Sobolev}}{\lesssim} \frac{r_*^d(x_0)}{R_0^{d-2}} \int_{\R^d} \a{\nabla u}^2 \overset{\eqref{E424}}\lesssim \frac{r_*^d(x_0)}{R_0^{d-4}} \int_{\R^d} \a{f}^2. 
\end{align*}
If otherwise $r_*(x_0) > R_0$, then we directly bound
\begin{align*}
 \a{ F(u) }^2 &\uoset{\textrm{Jensen},d\ge 3}{}{\le} \int_{B_1(x_0)} \a{ u }^2 \le \int_{B_{R_0}(x_0)} \a{u}^2 \lesssim r_*^d(x_0) \fint_{B_{R_0}(x_0)} \a { u}^2
\end{align*}
and proceed analogously to the other case.
Using the Green's function representation formula we have $u(x) = \int_{B_{R_0}(y_0)} G(x,y) f(y) \dy$, and thus the linearity of $F$ yields
\begin{equation}\nonumber
 \a{ F(u) } = \a{ \int_{B_{R_0}(y_0)} F( G(\cdot,y) ) f(y) \dy }.
\end{equation}
Since $f \in L^2(B_{R_0}(y_0))$ was arbitrary, we may combine the two previous estimates and conclude
\begin{equation}\label{E445}
 \int_{B_{R_0}(y_0)} \a{ F(G(\cdot,y)) }^2 \dy \lesssim \frac{r_*^d(x_0)}{R_0^{d-4}}.
\end{equation}
As before, it remains to argue that by going from $\int_{B_{R_0}(y_0)}$ to $\int_{B_1(y_0)}$ we gain a factor $R_0^{-d}$. We define $v(y) := F(G(\cdot,y))$, and observe that $-\nabla A^t \nabla v = 0$ in $B_{R_0}(y_0)$. %, where $A^t$ denotes the adjoint coefficient field (in coordinates this means $A_{ij}^{\alpha\beta} = (A^t)_{ji}^{\beta\alpha}$). Let $r_0'(y_0)$ denotes the minimal diameter from which $\fint_{B_r(y_0)} \a{ \nabla v}^2$ is non-decreasing. 

Now we use Lemma \ref{lm1} with $v$ to get 
 \begin{equation}\label{E451}
  \int_{B_1(x_0)} \a{ v }^2  \le \int_{B_{r_*'(y_0)}} \a{v}^2 \overset{\textrm{Lemma \ref{lm1}}}\lesssim \p{\frac{r_*'(y_0)}{R_0}}^d \int_{B_{R_0}(y_0)} \a{v}^2 \overset{\eqref{E445}}\lesssim \frac{\p{r_*(x_0)r_*'(y_0)}^d}{R_0^{2d-4}}.
 \end{equation}
For the choice $F(v) = \fint_{B_1(x_0)} v$, relation \eqref{E451} is exactly \eqref{E402}. 
%%%%%%%%%%
%%%%%%%%%%
%%%%%%%%%%

\section{Proof of Corollary \ref{cor1}}\label{sec2d}
We provide a generalization of (\ref{E52})-(\ref{E53}) in the two-dimensional case. When $d=2$, the Green's function for the whole space $\mathbb{R}^2$ does not have to exist; nevertheless, we may give a definition for $\nabla G$ via the Green's function on $\mathbb{R}^3$. To this purpose we introduce the following notation: If $\bar x \in \mathbb{R}^3$, we write $\bar x= (x, x_3) \in \mathbb{R}^2 \times \mathbb{R}$ and denote by $\bar B_r \subset \R^3$ and $B_r \subset \R^2$ the balls of radius~$r$ and centered at the origin. For a given bounded and uniformly elliptic coefficient field $A$ in $\mathbb{R}^2$, recall that its trivial extension $\bar A$ to $\mathbb{R}^3$ was defined in~\eqref{2d1} by
\begin{equation*}
 \bar A(x,x_3) := \begin{pmatrix} A(x) & 0\ \\ 0 & 1\ \end{pmatrix},
\end{equation*}
and the three-dimensional Green's function $\bar{G}=\bar{G}(\bar{A};\bar{x},\bar{y})$ is defined as a solution of 
$$
-\nabla_{\bar x} \cdot \bar{A}\nabla_{\bar x} \bar{G}(\bar A; \cdot , \bar y)=\delta (\cdot -\bar y).
$$
It will become clear below that the argument for the representation formula for $\nabla G$ through $\nabla_x \bar G$ calls for the notion of pointwise existence in $\bar y \in \mathbb{R}^3$ of the Green's function $\bar G (\bar A; \cdot , \bar y)$. As mentioned in Section 1, in the case of systems we may only rely on a definition of the Green's function for almost every singularity point $\bar y$. Therefore, differently from the previous sections, we need to bear in mind this weaker notion of existence of $\bar G$.

\bigskip

\noindent{\bf Step 1.} We argue that for almost every $y \in \mathbb{R}^2$ the function $\nabla G$ (since $G$ does not exist, $\nabla G$ should be understood as a symbol for a function and not as a gradient of some function $G$), defined through 
\begin{equation}\label{2d4}
\nabla{G(A;\cdot ,y)}:=\int_{\mathbb{R}}\nabla_x\bar{G}(\bar{A};(\cdot,x_3),(y,y_3)) \dx_3,
\end{equation}
satisfies for every $\zeta \in C^\infty_0(\mathbb{R}^2)$ 
\begin{align}\label{2d10}
\int \nabla_x \zeta(x) \cdot A(x) \nabla G(A; x, y) \dx = \zeta(y),
\end{align}
i.e., in a weak sense it solves $-\nabla_x \cdot A \nabla G(A; \cdot, y)= \delta(\cdot -y)$.\\

\noindent By definition of $\bar G(\bar A; \cdot, \cdot)$, we have for almost every $\bar y\in \mathbb{R}^3$ and every $\bar \zeta \in C^\infty_0( \mathbb{R}^3)$  
\begin{align}\nonumber
\int \nabla_{\bar x} \bar \zeta(\bar x) \cdot \bar A \nabla_{\bar x}\bar G(\bar A; \bar x, \bar y) \dbx = \bar \zeta(\bar y).
\end{align}
Thus, for any $\bar \rho \in C^\infty_0(\mathbb{R}^3)$ this yields
\begin{align}\nonumber
\int  \bar \rho(\bar y) \int \nabla_{\bar x}\bar \zeta(\bar x) \cdot \bar A \nabla_{\bar x}\bar G(A; \bar x, \bar y) \dbx \dby = \int \bar \rho(y) \bar \zeta(\bar y) \dby.
\end{align}
We now choose a sequence $\{\bar \zeta_n \}_{n \in \mathbb{N}}$ of test functions $\bar \zeta_n = \eta_n \zeta$, with $\zeta= \zeta(x) \in C^\infty_0( \mathbb{R}^2)$ and $\eta_n= \eta_n(x_3)$ smooth cut-off function for $\{|x_3|< n\}$ in $\{|x_3| < n+1 \}$: 
From the previous identity and definition (\ref{2d1}) it follows
\begin{align*}
\int \bar \rho(\bar y) \int &\zeta(x) \eta_n'(x_3) \partial_{x_3}\bar G(\bar A; \bar x, \bar y) \dbx \dby \\
&+ \int \bar \rho(\bar y) \int \eta_n(x_3) \nabla\zeta(x) \cdot A \nabla \bar G(\bar A; \bar x, \bar y) \dbx \dby = \int \bar \rho(y) \zeta(y) \dby.
\end{align*}
We now want to send $n \rightarrow +\infty$ in the previous identity : By our assumptions on $\bar \rho$ and $\bar \zeta_n$, if we show that
\begin{align}\label{2d9}
\int_{\textrm{supp}(\bar\rho)}\int_{\textrm{supp}(\zeta) \times \mathbb{R}} |\nabla_{\bar x} \bar G( \bar A; \bar x, \bar y)| \dbx \dby < +\infty,
\end{align}
then by the Dominated Convergence Theorem we may conclude that 
\begin{align}\nonumber
\int\bar \rho(\bar y) \int \nabla\zeta(x) \cdot A \bigl( \int_\R \nabla \bar G(\bar A; \bar x, \bar y) \dx_3 \bigr) \dx \dby = \int\bar \rho(\bar y) \zeta(y) \dby,
\end{align}
and thus (\ref{2d10}) by the arbitrariness of the test function $\bar \rho$ and the separability of $C^\infty_0(\mathbb{R}^2)$.

\medskip

\noindent To argue inequality (\ref{2d9}) we proceed as follows: We define a finite radius $M$ such that 
\begin{equation}\nonumber
  M \ge \max(r_*(\bar A^t,\bar X),r_*(A,\bar Y)) \quad \textrm{and} \quad \textrm{supp}(\bar \rho) \subset \bar B_{M}(\bar Y), \ \textrm{supp}(\zeta) \subset B_{M/2}(X),
 \end{equation}
and observe that inequality (\ref{2d9}) is implied by
\begin{align}\label{2d11}
 \int_{\bar B_{M}(\bar Y)} \int_{B_{M/2}(X) \times \mathbb{R}} |\nabla_{\bar x} \bar G| \dbx \dby < +\infty.
\end{align}
Since $\bar A$ is translational invariant in $x_3$, the minimal radius $r_*(\bar A^t,\cdot)$ is independent of $x_3$. Then, by the definition of $M$ and Remark~\ref{rmk1} we have 
\begin{align}\label{2d14}
\int_{\bar B_{M}(\bar  Y)}\int_{\bar B_{M}((X ,X_3))}  |\bar\nabla_x \bar G(\bar A; \bar x, \bar y)|^2 \dbx \dby \lesssim \frac{ M^6}{|Y - (X ,X_3)|^4} \le \frac{ M^6}{|Y_3 - X_3|^4}
\end{align}
provided $|X_3 - Y_3| \ge 3M$. 
\medskip

We now cover the cylinder $B_{M/2}(X) \times \mathbb{R}$ with countably many balls of radius $M$ centered at the points $(X, \pm Mn) \in \mathbb{R}^3$. By translational invariance we can w.\ l.\ o.\ g.\ assume that $Y_3=0$. We thus bound the integral in~\eqref{2d11} by
\begin{align}\nonumber
\int_{\bar B_M(\bar Y)} &\int_{B_{M/2}(X) \times \mathbb{R}} |\nabla_{\bar x} \bar G| \dbx \dby \leq \sum_{n=0}^{+\infty}\int_{\bar B_M(\bar Y)} \int_{\bar B_{M}(X,\pm Mn)} |\nabla_{\bar x} \bar G| \dbx \dby 
\\ \label{p102}
&\lesssim \int_{\bar B_M(\bar Y)} \int_{\bar B_{4M}((X,0))} |\nabla_{\bar x} \bar G| \dbx \dby + \sum_{n > 4}\int_{\bar B_M(\bar Y)} \int_{\bar B_{M}(X,\pm Mn)} |\nabla_{\bar x} \bar  G| \dbx \dby.
\end{align}
We claim that $\nabla_{\bar x} \bar G( \bar A; \cdot, \cdot) \in L^1_{loc}( \mathbb{R}^3 \times \mathbb{R}^3)$, and so the first integral on the r.\ h.\ s.\ of the previous identity is finite. 

Here we only sketch the idea why $\bar \nabla_{\bar x} \bar G \in L^1_{loc}(\R^3 \times \R^3)$; for the proof with all the details we refer to the proof of \cite[Theorem 1]{ConlonGiuntiOtto}. To show that $\bar \nabla_{\bar x} \bar G \in L^1_{loc}$ it suffices to show that $\int_{\bar B_R(0)} \int_{\bar B_R(0)} |\bar \nabla_{\bar x}\bar G| < \infty$. In order to do that we observe that for given two distinct points $\bar x, \bar y \in \R^3$, the proof of Theorem~\ref{thm1} (without the use of $r_*$ to go to smaller scales; see also Remark~\ref{rmk1}) implies in $3D$
\begin{equation*}
 \biggl( \int_{\bar B_r(\bar x)} \int_{\bar B_r(\bar y)} |\bar \nabla_{\bar x}\bar G|^2 \biggr)^{\frac{1}{2}} \lesssim \frac{|\bar B_r|}{r^2},
\end{equation*}
where $r = |\bar x-\bar y|/3$, which by H\"older's inequality turns into
\begin{equation*}
 \int_{\bar B_r(\bar x)} \int_{\bar B_r(\bar y)} |\bar \nabla_{\bar x}\bar G| \lesssim \frac{|\bar B_r|^2}{r^2}.
\end{equation*}
Using a simple covering argument, the above estimate holds also in the case when the balls are replaced by cubes. Since  $\bar B_R(0) \times \bar B_R(0)$ can be written as a null-set plus a countable union of pairs of open cubes $\bar Q_{r_n}(\bar x_n) \times \bar Q_{r_n}(\bar y_n)$, each with size $r_n := |\bar x_n - \bar y_n|/3$ and such that each pair of points $(\bar x,\bar y) \in \bar B_R(0) \times \bar B_R(0)$ belongs to at most one such pair of cubes, we conclude
\begin{equation}\nonumber
 \int_{\bar B_R(0)} \int_{\bar B_R(0)} |\bar \nabla_{\bar x}\bar G| \lesssim \int_{\bar B_{2R}(0)} \int_{\bar B_{2R}(0)} |\bar x- \bar y|^{-2} \dbx \dby < \infty,
\end{equation}
where we used that for $(\bar x,\bar y) \in \bar Q_{r_n}(\bar x_n) \times \bar Q_{r_n}(\bar y_n)$ we have $|\bar x - \bar y| \sim r_n$. 

Going back to the second term on the right-hand side of~\eqref{p102}, an application of H\"older's inequality in both variables $\bar x $ and $\bar y$ yields for the the sum over $n$ 
\begin{align*}
\sum_{n > 4}\int_{\bar B_M(\bar Y)} \int_{\bar B_{M}(X,\pm Mn)} |\nabla_{\bar x} \bar G|& \lesssim M^3 \sum_{n > 4} %n^{\frac 1 2} 
\biggl(\int_{\bar B_{M}(\bar Y)} \int_{\bar B_{M}(X,\pm M n)} |\bar \nabla_{\bar x} \bar G|^2 \biggr)^{\frac 1 2}.
\end{align*}
We now may apply to the r.h.s. the bound (\ref{2d14}) and thus obtain
\begin{align*}
\sum_{n > 4}\int_{\bar B_M(\bar Y)} \int_{\bar B_{M}(X,\pm Mn)} |\nabla_{\bar x} \bar G|&  \lesssim M^6 \sum_{n > 4} (Mn)^{-2} 
% \\ &
\lesssim M^4 < \infty.% \bigl(r_*(y_0) r_{*}'(x_0)\bigr)^{\frac 3 2},
\end{align*}
%which is finite by our hypothesis $r_*(y_0), r_*'(x_0) < +\infty$. Therefore, 
We have established (\ref{2d9}).

\medskip

\noindent Before concluding Step 1, we show that the representation formula (\ref{2d4}) does not depend on the choice of the coordinate $y_3 \in \mathbb{R}$, namely that for almost every two values $y_{0,3},y_{1,3} \in \mathbb{R}$, for almost every $y_0, x_0 \in \mathbb{R}^2$
\begin{align}\label{2d14b}
\int_{\mathbb{R}}\nabla_x \bar{G}(\bar{A};(x_0, x_3),(y_0, y_{0,3})) \dx_3&  = \int_{\mathbb{R}} \nabla_x\bar{G}(\bar{A};(x_0, x_3),(y_0, y_{1,3}) ) \dx_3.
\end{align}
Without loss of generality we assume $y_{0,3}=0$: Since by the uniqueness of $\bar G(\bar A; \cdot , \cdot)$, for every $\bar z \in \mathbb{R}^3$ and almost every $\bar x, \bar y \in \mathbb{R}^3$
$$
\bar G(\bar A;\bar x+\bar z , \bar y+\bar z) = \bar G(\bar A( \cdot + \bar z); \bar x, \bar y),
$$
by choosing $\bar z=(0, z_3)$ and using definition (\ref{2d1}) for $\bar A$, we get 
\begin{align}\label{2d13}
\bar G(\bar A;\bar x+\bar z , \bar y+\bar z) = \bar G(\bar A; \bar x, \bar y).
\end{align}
Let $x_0, y_0\in \mathbb{R}^2$ and $y_{1,3} \in \mathbb{R}^3$ be fixed: For every $\delta > 0$ we may write
\begin{align*}
\fint_{B_\delta (x_0)}&\fint_{\bar B_\delta((y_0, y_{1,3}))}\int_{\mathbb{R}}\nabla_x \bar{G}(\bar{A};\bar{x},\bar{y}) \dbx \dby
\\
&=\fint_{B_\delta (x_0)}\fint_{\bar B_\delta((y_0, y_{1,3}))}\int_{\mathbb{R}}\nabla_x \bar{G}(\bar{A};(x,x_3-y_{1,3}+y_{1,3}),(y, y_3-y_{1,3}+y_{1,3})) \dbx\dby,
\end{align*}
and use (\ref{2d13}) with $\bar z= (0, y_{1,3})$ to get
\begin{align*}
\fint_{B_\delta (x_0)}&\fint_{\bar B_\delta((y_0, y_{1,3}))}\int_{\mathbb{R}}\nabla_x \bar{G}(\bar{A};\bar{x},\bar{y}) \dbx\dby=\fint_{B_\delta (x_0)}\fint_{\bar B_\delta((y_0, 0))}\int_{\mathbb{R}}\nabla_x \bar{G}(\bar{A};\bar{x},\bar y) \dbx\dby.
\end{align*}
We now appeal to Lebesgue's theorem and conclude (\ref{2d14}).

\bigskip

\noindent{\bf Step 2.} Proof of (\ref{2d5}). For this part we denote $r_x := r_*(\bar A^t,(x_0,0))$ and $r_y := r_*(\bar A, (y_0,0))$. By translational invariance of $\bar A$ and $\bar A^t$ we have $r_x = r_*(\bar A^t,(x_0,x_3))$ and $r_y = r_*(\bar A,(y_0,y_3))$ for any $x_3, y_3 \in \R$. Denoting $\B := B_1(y_0) \times (-r_y/2,r_y/2)$, the independence of (\ref{2d4}) from $y_3$ yields
\begin{align*}
\int_{\B} \int_{B_1(x_0)}|\int_{\mathbb{R}}\nabla_{\bar{x}} \bar{G}(\bar{x},\bar{y})\dx_3|^2 \dx \dby 
&\uoset{\eqref{2d4}}{}{=} r_y \int_{B_1(y_0)} \int_{B_1(x_0)}|\int_{\mathbb{R}}\nabla_{\bar{x}} \bar{G}(\bar{x},(y, 0 ) )dx_3|^2 \dx \dy\\
&\stackrel{(\ref{2d4})}{=} r_y \int_{B_1(y_0)} \int_{B_1(x_0)} | \nabla G(A; x,y)|^2 \dx \dy.
\end{align*}
Since $\B \subset \bar B_{r_y}((y_0,0))$, the previous identity implies %we infer that for every $x_0, y_0 \in \mathbb{R}^2$ with $|x_0 - y_0 | \geq 3R_0 \geq 10$
\begin{equation} \nonumber %\label{p100}
\begin{aligned}
r_y \int_{B_1(x_0)}&\int_{B_1(y_0)}|\nabla_x G(A; x,y)|^2 \dx \dy = \int_{\B} \int_{B_1(x_0)}|\int_{\mathbb{R}}\nabla_x \bar{G}(\bar A;\bar{x},\bar{y}) \dx_3|^2 \dx \dby\\
&\lesssim  \int_{\bar B_{r_y}((y_0,0))} \int_{B_1(x_0)}|\int_{\mathbb{R}}\nabla_x \bar{G}(\bar A; \bar{x},\bar{y}) \dx_3|^2 \dx \dby\\
&\le \int_{\bar B_{r_y}((y_0,0))} \int_{B_1(x_0)} \biggl( \sum_{n=-\infty}^{\infty} \int_{nr_x}^{(n+1)r_x} |\nabla_x \bar{G}(\bar A; \bar{x},\bar{y})| \dx_3 \biggr)^2 \dx \dby.
\end{aligned}
\end{equation}
We define a sequence 
\begin{equation*}
a_n := \frac{(r_x r_y)^{\frac{3}{4}}}{( |x_0 - y_0|^2 + n^2 (r_x)^2)^{\frac{1}{2}}}
\end{equation*}
and observe that
\begin{align*}
\biggl( \sum_{n=-\infty}^\infty & \int_{nr_x}^{(n+1)r_x} |\nabla_x \bar G(\bar A;\bar x,\bar y)| \dx_3 \biggr)^2
= \biggl( \sum_{n=-\infty}^\infty a_n \frac{r_x}{a_n} \fint_{nr_x}^{(n+1)r_x} |\nabla_x \bar G(\bar A;\bar x,\bar y)|\dx_3 \biggr)^2
\\
&\overset{\textrm{H\"older}}{\le} \biggl( \sum_{n=-\infty}^\infty a_n^2 \biggr) \biggl( \sum_{n=-\infty}^\infty \frac{(r_x)^2}{a_n^2} \biggl( \fint_{nr_x}^{(n+1)r_x} |\nabla_x \bar G(\bar A;\bar x,\bar y)|\dx_3  \biggr)^2 \biggr)
\\
&\overset{\textrm{Jensen}}{\le} \biggl( \sum_{n=-\infty}^\infty a_n^2 \biggr) \biggl( \sum_{n=-\infty}^\infty \frac{r_x}{a_n^2} \int_{nr_x}^{(n+1)r_x} |\nabla_x \bar G(\bar A;\bar x,\bar y)|^2\dx_3  \biggr).
\end{align*}
Since
\begin{equation}\label{p101}
 \sum_{n=-\infty}^\infty a_n^2 \lesssim \frac{(r_x r_y)^{\frac{3}{2}}}{|x_0-y_0| r_x},
\end{equation}
where for simplicity we assumed $|x_0-y_0| \ge r_x$, we combine the three above relations to infer
\begin{align*}
r_y \int_{B_1(x_0)} &\int_{B_1(y_0)}|\nabla_x G(A; x,y)|^2 \dx \dy 
\\
&\uoset{\eqref{E53},d=3}{}{\lesssim} \frac{(r_x r_y)^{\frac{3}{2}}}{|x_0-y_0| r_x} \sum_{n} \frac{r_x}{a_n^2} 
\\
&\qquad \qquad \quad \times \int_{\bar B_{r_y}((y_0,0))} \int_{\bar B_{r_x}{(x_0,(n+1/2)r_x)}} |\nabla_x \bar G(\bar A;\bar x,\bar y)|^2 \dbx \dby 
\\
&\overset{\eqref{E53},d=3}{\lesssim} \frac{(r_x r_y)^{\frac{3}{2}}}{|x_0-y_0| r_x)} \sum_{n} \frac{r_x}{a_n^2} a_n^4 \overset{\eqref{p101}}{\lesssim} \frac{(r_x r_y)^3}{|x_0-y_0|^2 r_x},
\end{align*}
which is exactly \eqref{2d5}. 

%where $\bar x= ( x, x_3)$ and $\bar x' = (x, x_3')$. Therefore, 
% \begin{align*}
% \int_{B_1(x_0)}\int_{B_1(y_0)}&|\nabla_x G(A; x,y)|^2 \\
% &\lesssim \sum_{m,n} \int_{\bar B_1((y_0),0)} \int_{B_1(x_0)}\int_{n-1< |x_3|< n} \hspace{-0.5cm}\nabla_x \bar{G}(\bar A; \bar x, \bar y) \int_{m-1< |x_3'|< m} \hspace{-0.5cm} \nabla_x \bar{G}(\bar A; \bar x', \bar y).
% \end{align*}
% We now apply H\"older's inequality first in $x_3$ and $x_3'$ and then in $x$ and $\bar y$:
% \begin{align*}
% \int_{B_1(x_0)}\int_{B_1(y_0)}|\nabla_x G(a; x,y)|^2 &\lesssim \sum_{m,n} \bigl(\int_{\bar B_1((y_0,0))} \int_{B_1(x_0)}\int_{n-1< |x_3|< n} \hspace{-0.5cm}|\nabla_x \bar{G}(\bar A; \bar x, \bar y)|^2 \bigr)^{\frac 1 2}\\
% &\hspace{2cm}\times \bigl(\int_{\bar B_1((y_0,0))} \int_{B_1(x_0)}\int_{m-1< |x_3'|< m} \hspace{-0.5cm} |\nabla_x \bar{G}(\bar A; \bar x', \bar y)|^2 \bigr)^{\frac 1 2}\\
% &\simeq \biggl(\sum_{n} \bigl(\int_{\bar B_1((y_0,0))} \int_{B_1(x_0)}\int_{n-1< |x_3|< n} \hspace{-0.5cm}|\nabla_x \bar{G}(\bar A; \bar x, \bar y)|^2 \bigr)^{\frac 1 2} \biggr)^2\\
% &\lesssim \biggl(\sum_{n} \bigl(\int_{\bar B_1((y_0,0))} \int_{\bar B_2(x_0,\pm n)} \hspace{-0.5cm}|\nabla_x \bar{G}(\bar A; \bar x, \bar y)|^2 \bigr)^{\frac 1 2} \biggr)^2.
% \end{align*}
% By using (\ref{E52}) with $d=3$ and inequality (\ref{2d7}) we thus conclude
% \begin{align*}
% \int_{B_1(x_0)}\int_{B_1(y_0)}|\nabla_x G(A; x,y)|^2\lesssim \biggl( \sum_n \frac{\bigl(r_*'(x_0)r_*(y_0)\bigr)^{\frac{3}{2}}}{|x_0-y_0|^2+n^2}\biggr)^2\lesssim \frac{r_*'^{\ 3}(x_0)r_*^3(y_0)}{R_0^2}.
% \end{align*}

\medskip

\noindent %Exchanging the roles of x and y and considering in (\ref{2d1})-(\ref{2d4}) the adjoint of A, we obtain also (\ref{2d6}) analogously to (\ref{2d5}). 

Concerning (\ref{2d5b}), there are two possible ways how to proceed. For the first we observe that \eqref{2d10} implies for every test function $\phi \in C^\infty_c(\R^2)$
\begin{equation}\nonumber
 \int \nabla \phi(x) \cdot A(x) \biggl( \int \nabla_y \nabla G(x,y) \cdot f(y) \dy \biggr) \dx = \int \nabla \phi \cdot f = \int \nabla\phi \cdot A\nabla u,
\end{equation}
where $f \in L^2(\R^2;\R^2)$ and $u$ is a solution of $-\nabla \cdot A \nabla u = -\nabla \cdot f$. Therefore we have that
\begin{equation}
 \nabla u(x) = \int \nabla_y \nabla G(x,y) \cdot f(y) \dy,
\end{equation}
and the proof of~\eqref{E47} applies verbatim. A different way would be to mimic the argument for~\eqref{2d5}, i.e., to define $\nabla_y \nabla G$ as an integral of the second mixed derivative of the Green's function in three dimension. Unfortunatelly, this way we would obtain the estimate where the minimal radii in $2D$ appearing on the right-hand side of~\eqref{2d5b} would need to be replaced (with possibly larger) minimal radii for $3D$. 

% relying on Step 1 and definition (\ref{2d4}), we may write
% \begin{align*}
% \nabla_x\nabla_y G(A; x,y) = \int \nabla_x\nabla_y G(\bar A; \bar x, \bar y) dx_3 
% \end{align*}
% and estimate
% \begin{align*}
%  \int_{B_1(x_0)}\int_{B_1(y_0)}|\nabla_x\nabla_y G(A; x,y)|^2\lesssim  \biggl(\sum_{n} \bigl(\int_{\bar B_1((y_0,0))} \int_{\bar B_2(x_0,\pm n)} \hspace{-0.5cm}|\nabla_x\nabla_y \bar{G}(\bar A; \bar x, \bar y)|^2 \bigr)^{\frac 1 2} \biggr)^2.
% \end{align*}
% Using (\ref{E47}) and arguing as for (\ref{2d5}), we establish inequality (\ref{2d5b}).

\section*{Acknowledgment}

We warmly thank Felix Otto for introducing us into the world of stochastic homogenization and also for valuable discussions of this particular problem. 

\providecommand{\bysame}{\leavevmode\hbox to3em{\hrulefill}\thinspace}
\providecommand{\MR}{\relax\ifhmode\unskip\space\fi MR }
\providecommand{\MRhref}[2]{%
  \href{http://www.ams.org/mathscinet-getitem?mr=#1}{#2}
}
\providecommand{\href}[2]{#2}

% \bibliographystyle{amsplain}
% \bibliography{bella}

\end{document}